\documentclass[amstex,11pt,leqno]{amsart}
\usepackage{amsmath,amsfonts,amssymb,amsthm,enumerate,hyperref,multicol}
\newtheorem{lemma}{Lemma}
\newtheorem{theorem}{Theorem}

\numberwithin{equation}{section}
\begin{document}
\leftline{ \scriptsize \it  }
\title[Complex New Operators]
{Approximation by Complex Baskakov-Sz\'asz-Durrmeyer Operators in Compact Disks}
\maketitle
\begin{center}
{\bf Sorin G. Gal}\\
Department of Mathematics and Computer Science,
University of Oradea\\
Str. Universitatii No. 1, 410087 Oradea, Romania\\
e-mail address : galso@uoradea.ro \\

and

{\bf Vijay Gupta}\\
Department of Mathematics,
Netaji Subhas Institute of Technology\\
Sector 3 Dwarka, New Delhi-110078, India\\
e-mail address : vijaygupta2001@hotmail.com \\
\end{center}

\begin{abstract}
In the present paper, we deal with the complex Baskakov-Sz\'asz-Durrmeyer mixed operators and study Voronovskaja type results with quantitative estimates for these operators attached to analytic functions of exponential growth in $\mathbb{D}_{R}=\{z\in \mathbb{C} ; |z|<R\}$. Also, the exact order of approximation is found. The method used allows to construct complex Sz\'asz-type and Baskakov-type approximation operators without to involve the values on $[0, \infty)$.
\end{abstract}

\section{Introduction}

The study of approximation properties for the Sz\'asz type operators on $[0, +\infty)$ was well established in  \cite{szasz} and then generalized in various ways, see e.g. \cite{AG1}. Also, very recently, approximation properties for several real operators including the Sz\'asz-Durrmeyer operators are presented in the book \cite{vn}.
In order to approximate integrable functions on the positive real axis, in \cite{vgggs} it was proposed the modifications of the Baskakov operators with the weights of Sz\'{a}sz basis functions under the integral sign. These operators reproduce only constant functions. Ten years later
in \cite{PNA_AJM} it was proposed yet another sequence of the Baskakov-Sz\'{a}sz-Durrmeyer operators which preserve constant as well as linear functions. Also, generalizations of the Durrmeyer polynomials were studied in e.g. \cite{AG2}.

In the complex domain, the overconvergence phenomenon holds, that is the extension of approximation properties from real domain to complex domain. In this context, the first qualitative kind results were obtained in the papers \cite{DGP1}, \cite{WR}, \cite{Wood1}. Then,
in the books \cite{Gal}, \cite{og} quantitative approximation results are presented for several type of approximation operators. For Sz\'{a}sz-Mirakjan operator and its Stancu variant in complex domain, we refer the readers to \cite{Aydin}, \cite{Ce_Is}, \cite{Ispir1}, \cite{G1}, \cite{Mah1}, \cite{Mah2}, \cite{Su_Ib} and \cite{VGDKV}. Also for complex Bernstein-Durrmeyer operators, several papers are available in the literature (see e.g. \cite{G2}, \cite{G3}, \cite{GaGu1}, \cite{GaGu2}, \cite{G6}, \cite{Mah3}, \cite{Os}, \cite{Ren1}, \cite{Ren2}),
for complex Sz\'{a}sz-Durrmeyer operators see \cite{CS}, while for complex $q$-Bal\'asz-Szabados operators see \cite{Ispir2}.

In the present paper, we study the rate of approximation of analytic functions in a disk $\mathbb{D}_{R}=\{z\in \mathbb{C} ; |z|<R\}$, i.e. $f(z)=\sum_{k=0}^{\infty}c_{k}z^{k}$, of exponential growth, and the Vo\-ro\-nov\-ska\-ja type result, for a natural derivation from the complex  operator $L_{n}(f)(z)$ introduced in the case of real variable in \cite{PNA_AJM}, and formally defined as operator of complex variable by
\begin{equation}\label{e:PNA_AJM1}
L_n(f)(z):=n\sum_{v=1}^\infty b_{n,v}(z)\int_0^\infty s_{n,v-1}(t)f(t)dt+(1+z)^{-n}f(0), z\in \mathbb{C},
\end{equation}
where
$$b_{n, k}(z)=\left(
\begin{array}{c}
n+k-1 \\
k
\end{array}
\right)\frac{z^k}{(1+z)^{n+k}}, s_{n,k}(t)=e^{-n t}\frac{(nt)^k}{k!}.$$
An important relationship used for the quantitative results in approximation of an analytic function $f$ by the complex operator $L_{n}(f)$ would be $L_{n}(f)(z)=\sum_{k=0}^{\infty}c_{k}L_{n}(e_{k})(z)$, but which requires some additional hypothesis on $f$ (because the definition of $L_{n}(f)(z)$ involves the values of $f$ on $[0, +\infty)$ too) and implies restrictions on the domain of convergence. This situation can naturally be avoided, by defining directly the approximation complex operator
$$L_{n}^{*}(f)(z)=\sum_{k=0}^{\infty}c_{k}\cdot L_{n}(e_{k})(z),$$
whose definition evidently that omits the values of $f$ outside of its disk of analyticity.

In this paper we deal with the approximation properties of the complex operator $L_{n}^{*}(f)(z)$.

It is worth noting here that if instead of the above defined $L_{n}(f)(z)$ we consider any other Sz\'asz-type or Baskakov-type complex operator, then
for $L_{n}^{*}(f)(z)$ defined as above, all the quantitative estimates in e.g. \cite{Aydin}, \cite{Ce_Is}, \cite{G1}, \cite{Gal4}, \cite{CS}, \cite{Gal5}, \cite{VGDKV}, \cite{Ispir1}, \cite{Mah1} hold true identically, without to need the additional hypothesis on the values of $f$ on $[0, \infty)$ imposed there.

Everywhere in the paper we denote $\|f\|_{r}=\max\{|f(z)| ; |z|\le r\}$.

\section{Auxiliary Result}
In the sequel, we need the following lemma :
\begin{lemma} \label{l:2} Denoting $e_{k}(z)=z^{k}$ and $T_{n, k}(z)=L_{n}(e_{k})(z)$, we have the recurrence formula
$$T_{n, k+1}(z)=\frac{z(1+z)}{n}T^{\prime}_{n, k}(z)+\frac{nz+k
}{n}T_{n, k}(z).$$
Also, $T_{n, k}(z)$ is a polynomial of degree $k$.
\end{lemma}
\begin{proof} Using $z(1+z)b_{n,\nu}^\prime(z)=(\nu-nz)b_{n,\nu}(z)$, we have
\begin{eqnarray*}z(1+z)T_{n, k}^\prime(z)&=&n\sum_{\nu=1}^{\infty}z(1+z)b_{n, \nu}^\prime(z)\int_{0}^{\infty}s_{n,\nu-1}(t)t^{k}d t \\
&=&n\sum_{\nu=1}^{\infty}(\nu-n z)b_{n, \nu}(z)\int_{0}^{\infty}s_{n,\nu-1}(t)t^{k}d t\\
&=&n\sum_{\nu=1}^{\infty}b_{n, \nu}(z)\int_{0}^{\infty}[(\nu-1-nt)+(1+nt-n z)]s_{n,\nu-1}(t)t^{k}d t\\
&=&n\sum_{\nu=1}^{\infty}b_{n, \nu}(z)\int_{0}^{\infty}s_{n,\nu-1}^\prime(t)t^{k+1}d t\\
&&+(1-n z)T_{n,k}(z)+nT_{n,k+1}(z)
\end{eqnarray*}
Thus integrating by parts the last integral, we get.
\begin{eqnarray*}z(1+z)T_{n, k}^\prime(z)&=&-(k+1)T_{n,k}(z)+(1-n z)T_{n,k}(z)+nT_{n,k+1}(z)
\end{eqnarray*}
which completes the proof of the recurrence relation. Taking above step by step $k=1, 2, ..., $, by mathematical induction we easily get that $T_{n, k}(z)$ is a polynomial of degree $k$.
\end{proof}

\section{Main Results}
Our first main result is the following theorem for upper bound.

\begin{theorem} \label{t:1}
For $f:\mathbb{D}_{R}\to \mathbb{C}$, $1<R<+\infty$, analytic on $\mathbb{D}_{R}$, i.e. $f(z)=\sum_{k=0}^{\infty}c_{k}z^{k}$, for all $z\in \mathbb{D}_{R}$, suppose that there exist $M>0$ and $A\in(\frac{1}{R},1)$, with the property that $|c_k|\le M\frac{A^k}{k!},$ for all $k=0,1,...,$ (which implies $|f(z)|\leq Me^{A|z|}$ for all $z\in \mathbb{D}_{R}$).

(i) If $1\le r <\frac{1}{A}$, then for all $|z|\le r$ and $n\in \mathbb{N}$ with $n > r+2$, $L_{n}^{*}(f)(z)$ is well-defined and we have
\begin{eqnarray*}
|L_{n}^{*}(f)(z)-f(z)|\le \frac{C_{r, A, M}}{n},
\end{eqnarray*}
where $C_{r,A, M}=\frac{M(r+2)}{r}\cdot \sum_{k=2}^\infty (k+1)\cdot (r A)^k <\infty ;$

(ii) If $1\le r < r_{1} < \frac{1}{A}$, then for all $|z|\le r$ and $n,
p\in \mathbb{N}$ with $n>r+2$, we have
$$|[L_{n}^{*}(f)]^{(p)}(z)-f^{(p)}(z)|\le \frac{p!r_{1}C_{r_{1}, A, M}}{n (r_{1}-r)^{p+1}},$$
where $C_{r_1, A, M}$ is given as at the above point (i).
\end{theorem}
\begin{proof} (i) By using the recurrence relation of Lemma \ref{l:2}, we have
$$T_{n, k+1}(z)=\frac{z(1+z)}{n}T^{\prime}_{n, k}(z)+\frac{nz+k}{n}T_{n, k}(z),$$
for all $z\in\mathbb{C},k\in\{0,1,2,....\},n\in N$. From this we immediately get the recurrence formula
\begin{eqnarray*}
T_{n, k}(z)-z^k&=&\frac{z(1+z)}{n}[T_{n, k-1}(z)-z^{k-1}]^\prime+\frac{nz+k-1}{n}[T_{n, k-1}(z)-z^{k-1}]\\
&&+\frac{(k-1)}{n}z^{k-1}(2+z),
\end{eqnarray*}
for all $z\in\mathbb{C},k,n\in N$.
Now for $1\leq r<R$, if we denote the norm-$||\cdot||_r$ in $C(\overline{\mathbb{D}}_r)$, where $\overline{\mathbb{D}}_r=\{z\in \mathbb{C}:|z|\leq r\}$, then by a linear transformation, the Bernstein's inequality in the closed unit disk becomes $|P_k^\prime (z)|\leq \frac{k}{r}||P_k||_r$, for all $|z|\leq r$, where $P_k(z)$ is a polynomial of degree $\leq k$. Thus from the above recurrence relation, we get
\begin{eqnarray*}
||T_{n, k}-e_k||_r&\le&\frac{r(1+r)}{n}\cdot ||T_{n, k-1}-e_{k-1}||_r\frac{k-1}{r}+\frac{nr+k-1}{n}||T_{n, k-1}-e_{k-1}||_r\\
&&+\frac{(k-1)}{n}(2+r)r^{k-1},
\end{eqnarray*}
which, by using the notation $\rho=r+2$, implies
$$||T_{n, k}-e_k||_r \le \left(r+\frac{(2+r)(k-1)}{n}\right)\cdot ||T_{n, k-1}-e_{k-1}||_r+\frac{(k-1)}{n}(2+r)r^{k-1}$$
$$=\left(r+\frac{\rho(k-1)}{n}\right)\cdot ||T_{n, k-1}-e_{k-1}||_r+\frac{(k-1)}{n}\rho\cdot r^{k-1}.$$
In what follows we prove by mathematical induction with respect to $k$ that for $n\ge \rho$, this recurrence implies
\begin{eqnarray*}
||T_{n, k}-e_k||_r&\le&\frac{\rho \cdot (k+1)!}{n}\cdot r^{k-1} \ \  \mbox{ for all} \ \ k\ge 1.
\end{eqnarray*}
Indeed for $k=1$ it is trivial, as the left-hand side is zero. Suppose that it is valid for $k$, the above recurrence relation implies that
\begin{eqnarray*}
||T_{n, k+1}-e_{k+1}||_r&\le&\left(r+\frac{\rho\cdot k}{n}\right)\cdot \frac{\rho\cdot (k+1)!}{n}r^{k-1}+\frac{\rho \cdot k}{n}r^{k}.
\end{eqnarray*}
It remains to prove that
\begin{eqnarray*}
\left(r+\frac{\rho\cdot k}{n}\right)\cdot \frac{\rho \cdot (k+1)!}{n}r^{k-1}+\frac{\rho\cdot k}{n}r^{k}\le \frac{\rho \cdot (k+2)!}{n}r^k,
\end{eqnarray*}
or after simplifications, equivalently to
\begin{eqnarray*}
\left(r+\frac{\rho \cdot k}{n}\right)\cdot (k+1)!+r k \le (k+2)! \cdot r,
\end{eqnarray*}
for all $k\in \mathbb{N}$ and $r\ge 1$.

Since by $n\ge \rho$, we get
$$\left(r+\frac{\rho\cdot k}{n}\right)\cdot (k+1)!+r k \le \left(r+k\right)\cdot (k+1)!+r k,$$
it is good enough if we prove that
$$\left(r+k\right)\cdot (k+1)!+r k \le (k+2)! \cdot r.$$
But this last inequality is obviously valid for all $k\ge 1$ (and fixed $r\ge 1$).

In conclusion, the required estimate holds.

Now, let us prove that $L_{n}^{*}(f)(z)$ is well-defined for $|z|\le r$ and $n > r+2$. Indeed, we have
$$|L_{n}^{*}(f)(z)|\le \sum_{k=0}^\infty |c_k|\cdot | L_n(e_k)(z)|=\sum_{k=0}^\infty |c_k|\cdot |T_{n,k}(z)|$$
$$\le M\cdot \sum_{k=0}^{\infty} \frac{A^{k}}{k !}(\|T_{n,k}-e_{k}-e_{k}\|_{r}+r^{k})\le Me^{A r}+
M\sum_{k=0}^\infty \frac{A^{k}}{k !}\cdot \frac{(r+2)(k+1)!}{n}\cdot r^{k}$$
$$\le Me^{A r}+ M(r+2)\sum_{k=0}^\infty (A r)^{k}(k+1)<+\infty.$$
In conclusion, we get
\begin{eqnarray*}
|L_n^{*}(f)(z)-f(z)|&\le& \sum_{k=2}^\infty |c_k|\cdot |T_{n,k}(z)-e_k(z)|\le \sum_{k=2}^\infty M\frac{A^k}{k!}\frac {(r+2)\cdot (k+1)!}{n}r^{k-1}\\
&=& \frac{M(r+2)}{r n}\sum_{k=2}^\infty (k+1)\cdot (r A)^k=\frac{C_{r, A, M}}{n},
\end{eqnarray*}
where $C_{r,A, M}=\frac{M(r+2)}{r}\cdot \sum_{k=2}^\infty (k+1)\cdot (r A)^k <\infty$ for all $1\le r <\frac{1}{A},$ taking into account that the series $\sum_{k=1}^\infty (k+1)u^k$ is uniformly convergent in any compact disk included in the open unit disk.

(ii) Denote by $\gamma$ the circle of radius $r_{1}>r$ and center $0$.
For any $|z|\le r$ and $v\in \gamma$, we have $|v-z|\ge r_{1}-r$ and by the Cauchy's formula, for all $|z|\le r$ and $n > r+2$ it follows
\begin{eqnarray}
|[L_{n}^{*}(f)]^{(p)}(z)-f^{(p)}(z)|&=& \frac{p!}{2\pi}\left
|\int_{\gamma}\frac{L_{n}^{*}(f)(v)-f(v)}{(v-z)^{p+1}}dv\right |
\le  \frac{C_{r_{1}, A, M}}{n}\frac{p!}{2\pi}\frac{2\pi r_{1}}{(r_{1}-r)^{p+1}} \nonumber \\
&=& \frac{C_{r_{1}, A, M}}{n}\frac{p! r_{1}}{(r_{1}-r)^{p+1}}, \nonumber
\end{eqnarray}
which proves (ii) and the theorem.
\end{proof}
The following Voronovskaja type result holds.
\begin{theorem}\label{t:2} For $f:\mathbb{D}_{R}\to \mathbb{C}$, $2<R<+\infty$, analytic on $\mathbb{D}_{R}$, i.e. $f(z)=\sum_{k=0}^{\infty}c_{k}z^{k}$, for all $z\in \mathbb{D}_{R}$, suppose that there exist $M>0$ and $A\in(\frac{1}{R},1)$, with the property that $|c_k|\le M\frac{A^k}{k!},$ for all $k=0,1,...,$.

If $1\le r< r+1<\frac{1}{A}$ then for all $|z|\le r$ and $n\in \mathbb{N}$ with $n>r+2$, we have
$$\left|L_{n}^{*}(f)(z)-f(z)-\frac{z(z+2)}{2n}f^{\prime \prime}(z)\right|\leq \frac{C_{r, A, M}(f)}{n^2},$$
where $C_{r, A, M}(f)=M\sum_{k=2}^\infty \frac{k-1}{k !}[A(r+1)]^{k} B_{k,r}+\frac{4 M(r+2)}{r}\cdot \frac{1}{ln^{2}(1/\rho)}\cdot \left (\frac{1}{(1-A r)^{2}}+\frac{4}{1-A r}\right )<\infty$ and $$B_{k,r}=(k-1)^2(k-2)r^2+2(k-1)(k-2)(2k-3)(r+1)+(r+1)(r+2)\cdot (k+1)!.$$
\end{theorem}
\begin{proof} By Theorem \ref{t:1}, (i), $L_{n}^{*}(f)(z)$ is well-defined, for all $|z|\le r$, $n > r+2$. We can write
$$\frac{z(z+2)f^{\prime \prime}(z)}{2n}=\frac{z(z+2)}{2n}\sum_{k=2}^\infty c_kk(k-1)z^{k-2}
=\frac{1}{2n}\sum_{k=1}^\infty c_kk(k-1)(z+2)z^{k-1}.$$
Thus
$$\left|L_n^{*}(f)(z)-f(z)-\frac{z(z+2)}{2n}f^{\prime \prime}(z)\right|
\leq \sum_{k=1}^\infty |c_k|\left|T_{n,k}(z)-e_k(z)-\frac{k(k-1)(z+2)z^{k-1}}{2n}\right|.$$

By Lemma \ref{l:2}, for all $n\in \mathbb{N}, z\in \mathbb{C}$ and $k=0,1,2,...$, we have
$$T_{n, k+1}(z)=\frac{z(1+z)}{n}T^{\prime}_{n, k}(z)+\frac{nz+k
}{n}T_{n, k}(z).$$
If we denote
$$E_{k,n}(z)=T_{n,k}(z)-e_k(z)-\frac{k(k-1)(z+2)z^{k-1}}{2n},$$
then it is obvious that $E_{k,n}(z)$ is a polynomial of degree less than or equal to $k$ and by simple computation and the use of above recurrence relation, we are led to
$$E_{k,n}(z)=\frac{z(1+z)}{n}E_{k-1,n}^{\prime}(z)+\frac{nz+k-1}{n}E_{k-1,n}(z)+X_{k,n}(z),$$
where
$$X_{k,n}(z)=\frac{z^{k-2}}{2n^2}\left[(k-1)^2(k-2)z^2+2(k-1)(k-2)(2k-3)z+2(k-1)(k-2)(2k-3)\right]$$
for all $k\geq 2,n\in \mathbb{N}$ and $|z|\le r$.

Using the estimate in the proof of Theorem \ref{t:1},  we have
$$|T_{n,k}(z)-e_k(z)|\leq \frac{(k+1)! (r+2) r^{k-1}}{n},$$
for all $k\ge 1, n\ge r+2, |z|\leq r$, with $1\leq r$.

It follows
$$|E_{k,n}(z)|\leq \frac{r(r+1)}{n}|E_{k-1,n}^{\prime}(z)|+\left (r+\frac{k-1}{n}\right )|E_{k-1,n}(z)|+|X_{k,n}(z)|.$$
Now we shall find the estimation of $|E_{k-1,n}^{\prime}(z)|$ for $k\geq 2$. Taking into account the fact that $E_{k-1,n}(z)$ is a polynomial of degree $\leq k-1$, we have
$$|E_{k-1,n}^{\prime}(z)|\leq \frac{k-1}{r}||E_{k-1,n}||_r$$
$$\leq \frac{k-1}{r}\left[||T_{n,k-1}(z)-e_{k-1}(z)||_r+\left|\left|\frac{(k-1)(k-2)[e_1+2]e_{k-2}}{2 n}\right|\right|_r\right]$$
$$\leq \frac{k-1}{r}\left[\frac{k! (r+2) r^{k-2}}{n}+\frac{r^{k-2}(k-1)(k-2)(r+2)}{n}\right]$$
$$ \leq \frac{(r+2) r^{k-2} (k-1)(k! + (k-1)(k-2))}{nr}.$$
Thus, by the obvious inequality $(k-1)(k! + (k-1)(k-2))\le (k+1)!$, we get
$$\frac{r(1+r)}{n}|E_{k-1,n}^{\prime}(z)|\leq \frac{(r+1)(r+2)\cdot (k+1)! r^{k-2}}{n^2}$$
and
$$|E_{k,n}(z)|\leq  \frac{(r+1)(r+2)\cdot (k+1)! r^{k-2}}{n^2}+\left (r+\frac{k-1}{n}\right )|E_{k-1,n}(z)|+|X_{k,n}(z)|,$$
for all $|z|\leq r, k\geq 2$ and $n > r+2$.

For $ k-1\le n$ (i.e. $k\le n+1$) and $|z|\le r$, taking into account that $r+(k-1)/n\le r+1$, we get
$$|E_{k,n}(z)|\leq \frac{(r+1)(r+2)(k+1)!\cdot r^{k-2}}{n^2}+(r+1)|E_{k-1,n}(z)|+|X_{k,n}(z)|,$$
where
$$|X_{k,n}(z)|\le \frac{r^{k-2}}{n^2}\left[(k-1)^2(k-2)r^2+2(k-1)(k-2)(2k-3)r+2(k-1)(k-2)(2k-3)\right]$$
$$\le \frac{r^{k-2}}{n^2} A_{k,r},$$
for all $|z|\le r, k\ge 1, n>r+2,$ where
$$A_{k,r}=(k-1)^2(k-2)r^2+2(k-1)(k-2)(2k-3)(r+1).$$
Thus for all $|z|\le r, n > r+2$ and $k\le n+1$,
$$|E_{k,n}(z)|\le (r+1)|E_{k-1,n}(z)|+\frac{r^{k-2}}{n^2}B_{k,r},$$
where $$B_{k,r}=A_{k,r}+(r+1)(r+2)\cdot (k+1)!.$$
But $E_{0,n}(z)=E_{1,n}(z)=0$, for any $z\in \mathbb{C}$ and therefore by writing last inequality for $2\le k\le n+1$, we easily obtain step by step the following
$$|E_{k,n}(z)|\leq \frac{(r+1)^{k}}{n^2}\sum_{j=2}^{k} B_{j,r}\le \frac{(k-1)(r+1)^k}{n^2}B_{k,r}.$$
It follows that
$$\left|L_n^{*}(f)(z)-f(z)-\frac{z(z+2)}{2n}f^{\prime \prime}(z)\right|\leq \sum_{k=2}^{n+1} |c_k|\cdot |E_{k,n}(z)|+\sum_{k=n+2}^{\infty} |c_k|\cdot |E_{k,n}(z)|$$
$$\le \frac{1}{n^2}\sum_{k=2}^{\infty}|c_k|(k-1)(r+1)^kB_{k,r}
+\sum_{k=n+2}^{\infty} |c_k|\cdot \left[|T_{n,k}(z)-e_k(z)|+\frac{k(k-1)(r+2)r^{k-1}}{2n}\right]$$
$$\le \frac{1}{n^2}\sum_{k=2}^{\infty}|c_k|(k-1)(r+1)^k B_{k,r}
+\sum_{k=n+2}^{\infty} |c_k|\cdot  \left[\frac{(r+2)\cdot (k+1)!}{n}\cdot r^{k-1}+\frac{k(k-1)(r+2)r^{k-1}}{2n}\right]$$
$$\le \frac{1}{n^2}\sum_{k=2}^{\infty}|c_k|(k-1)(r+1)^k B_{k,r}
+2\sum_{k=n+2}^{\infty} |c_k|\cdot \frac{(r+2)\cdot (k+1)!}{n}\cdot r^{k-1}$$
$$\le \frac{1}{n^2}\sum_{k=2}^{\infty}|c_k|(k-1)(r+1)^k B_{k,r}
+\frac{2 M(r+2)}{n r}\sum_{k=n+2}^{\infty} (k+1)(A r)^k$$
$$=\frac{1}{n^2}\sum_{k=2}^{\infty}|c_k|(k-1)(r+1)^k B_{k,r}
+\frac{2 M(r+2)}{n r}\cdot (A r)^{n+2}\sum_{k=n+2}^{\infty} (k+1)(A r)^{k-n-2}.$$
But, denoting for simplicity $\rho=A r <1$, we easily obtain
$$(A r)^{n+2}\cdot \sum_{k=n+2}^{\infty} (k+1)(A r)^{k-n-2}=\rho^{n+2}\cdot\left [\sum_{j=0}^{\infty}j \rho^{j}+(n+3)\sum_{j=0}^{\infty}\rho^{j}\right ]$$
$$=\rho^{n+2}\left [\rho \cdot \left (\sum_{j=0}^{\infty}\rho^{j}\right )^{\prime}+(n+3)\cdot \frac{1}{1-\rho}\right ]
=\rho^{n+3}\cdot \frac{1}{(1-\rho)^{2}}+\rho^{n+2}\cdot (n+3)\cdot \frac{1}{1-\rho},$$
which leads to the estimate
$$\left|L_n^{*}(f)(z)-f(z)-\frac{z(z+2)}{2n}f^{\prime \prime}(z)\right|$$
$$\leq \frac{1}{n^2}\sum_{k=2}^{\infty}|c_k|(k-1)(r+1)^k B_{k,r}+\frac{2M(r+2)}{n r}\cdot \left [\frac{(A r)^{n+3}}{(1-A r)^{2}}+\frac{(A r)^{n+2} (n+3)}{1-A r}\right ]$$
$$\le M \frac{1}{n^2}\sum_{k=2}^{\infty}\frac{k-1}{k!}[A(r+1)]^k \cdot B_{k,r}
+\frac{4 M (r+2)}{r n^{2}}\cdot \frac{1}{ln^{2}(1/(Ar))}\cdot \left [\frac{1}{(1-A r)^{2}}+\frac{4}{1-A r}\right ],$$
where we used the inequality
$$\rho^{n+3}\le \rho^{n+2}\le \rho^{n}\le \frac{2}{ln^{2}(1/(A r))}\cdot \frac{1}{n^{2}}, \mbox{ for all } n\in \mathbb{N},$$
applied for $\rho=A r<1$ and where for $(r+1)A<1$, we obviously have the series $\sum_{k=2}^{\infty}|c_k|(k-1)(r+1)^kB_{k,r}$ convergent.

Indeed, by $e^{x}=1+x+\frac{x^{2}}{2}+\frac{x^{3}}{6}+ ... $, we get $e^{x}\ge \frac{x^{2}}{2}$, for all $x\ge 0$. Then, by $(1/\rho)^{n}=e^{n ln(1/\rho)}$, it follows $\frac{1}{\rho^{n}}\ge \frac{n^{2} ln^{2}(1/\rho)}{2}$, for all $n\in \mathbb{N}$, which immediately implies the above claimed inequality.
\end{proof}
The following exact order of approximation can be obtained.

\begin{theorem}\label{t:3}
Suppose that the hypothesis in Theorem \ref{t:2} hold.

(i) If $f$ is not a polynomial of degree $\le 1$ then for all $n > r+2$ we have
$$\|L_{n}^{*}(f)-f\|_{r}\sim \frac{1}{n}.$$
where the constants in the equivalence depend only on $f$, $R$, $A$ and $r$.

(ii) If $1\le r < r_{1} < r_{1}+ 1 < 1/A$ and $f$ is not a polynomial of degree $\le p-1, (p\ge 1)$ then
$$\|[L_{n}^{*}(f)]^{(p)}-f^{(p)}\|_{r}\sim \frac{1}{n}, \mbox{ for all } n > r+2,$$
where the constants in the equivalence depend only on $f$, $R$, $A$, $r$, $r_{1}$ and $p$.
\end{theorem}
\begin{proof} (i) For all $|z|\le r$ and $n\in\mathbb{N}$ with $n > r+2$, we can write
\begin{eqnarray*}
L_n^{*}(f)(z)-f(z)&=&\frac{1}{n}\bigg[\frac{z(z+2)}{2}f^{\prime \prime}(z) \\
&&+\frac{1}{n}\cdot n^{2}\left(L_n^{*}(f)(z)-f(z)-\frac{z(z+2)}{2n}f^{\prime \prime}(z)\right)\bigg]
\end{eqnarray*}
Applying the inequality
\begin{eqnarray*}
\|F+G\|_{r}\ge|\,\|F\|_{r}-\|G\|_{r}\,|\ge\|F\|_{r}-\|G\|_{r},
\end{eqnarray*}
we obtain
$$\|L_n^{*}(f)-f\|_{r}\ge\frac{1}{n}\left [\left\|\frac{e_1(e_1+2)}{2}f^{\prime \prime}\right\|_{r}\right .
\left .- \frac{1}{n}\cdot n^{2}\left\|L_n^{*}(f)-f-\frac{e_1(e_1+2)f^{\prime \prime}}{2n}\right\|_{r}\right ].$$
Since $f$ is not a polynomial of degree $\le 1$, we get $\left\|\frac{e_1(e_1+2)}{2}f^{\prime \prime}\right\|_{r}>0$. Indeed, supposing the contrary, it follows that
$$\frac{z(z+2)}{2}f^{\prime \prime}(z)=0,\, \mbox{ for all } |z|\le r.$$
The last equality is equivalent to $f(z)=C_1z+C_2$, with $C_1, C_2$ are constants, a contradiction with the hypothesis.
Now by Theorem \ref{t:2}, we have
$$n^{2}\left\|L_n^{*}(f)-f-\frac{e_1(e_1+2)f^{\prime \prime}}{2n}\right\|_{r}\le C_{r, A, M}(f),$$
for all $n > r+2$.

Thus, there exists $n_{0} > r+2$ (depending on $f$ and $r$ only) such that for all $n\ge n_0,$ we have
$$\left\|\frac{e_1(e_1+2)}{2}f^{\prime \prime}\right\|_{r}
\left .- \frac{1}{n}\cdot n^{2}\left\|L_n^{*}(f)-f-\frac{e_1(e_1+2)f^{\prime \prime}}{2n}\right\|_{r}\right ]
\ge \frac{1}{4}\left\|e_1(e_1+2)f^{\prime \prime}\right\|_{r},$$
which implies that
\begin{eqnarray*}
\|L_n^{*}(f)-f\|_{r}\ge \frac{1}{4 n}\left\|e_1(e_1+2)f^{\prime \prime}\right\|_{r},
\end{eqnarray*} for all $n\ge n_0$.\\
For $r+2 < n \le n_{0}-1$, we get $\|L_n^{*}(f)-f\|_{r}\ge\frac{M_{r,n}(f)}{n}$ with $M_{r,n}(f)=n\cdot \|L_n^{*}(f)-f\|_{r}>0$ (since
$\|L_n^{*}(f)-f\|_{r}=0$ for a certain $n$ is valid only for $f$ a polynomial of degree $\le 1$, contradicting the hypothesis on $f$).

Therefore, finally we have
$$||L_n^{*}(f)-f||_{r}\ge \frac{C_r(f)}{n}$$ for all $n>r+2$, where
$$C_r(f)=\min_{n_{0}-1\ge n > r+2}\left\{M_{r,n}(f),....,M_{r,n_{0}-1}(f),\frac{1}{4}\left\|e_1(e_1+2)f^{\prime \prime}\right\|_{r}\right\},$$
which combined with Theorem \ref{t:1}, (i), proves the desired conclusion.

(ii) The upper estimate is exactly Theorem \ref{t:1}, (ii), therefore it remains to prove the lower estimate.
Denote by $\gamma$ the circle of radius $r_{1}>r$ and center $0$.
For any $|z|\le r$ and $v\in \gamma$, we have $|v-z|\ge r_{1}-r$ and by the Cauchy's formula, for all $|z|\le r$ and $n > r+2$ it follows
$$[L_{n}^{*}(f)]^{(p)}(z) - f^{(p)}(z)=\frac{p!}{2\pi i}\int_{\gamma}\frac{L_{n}^{*}(f)(v)-f(v)}{(v-z)^{p+1}}dv,$$
where $|v-z|\ge r_{1}-r$, for all $v\in \gamma$.

Since for $v\in \gamma$ we get
$$L_{n}^{*}(f)(v)-f(v)$$
$$=\frac{1}{n}\left \{\frac{v(v+2)}{2}f^{\prime \prime}(v)+
\frac{1}{n}\left [n^{2}\left (L_{n}^{*}(f)(v)-f(v)-\frac{v(v+2)f^{\prime \prime}(v)}{2n}\right )\right ]\right \},$$
replaced in the Cauchy's formula implies
$$[L_{n}^{*}(f)]^{(p)}(z) - f^{(p)}(z)=\frac{1}{n}\left \{\frac{p!}{2\pi i}\int_{\gamma}\frac{v(v+2)f^{\prime \prime}(v)}{2(v-z)^{p+1}}dv\right .$$
$$+\left .\frac{1}{n}\cdot \frac{p!}{2\pi i}\int_{\gamma}\frac{n^{2}\left (L_{n}^{*}(f)(v)-f(v)-\frac{v(v+2)f^{\prime \prime}(v)}{2n}\right )}{(v-z)^{p+1}}dv\right \}$$
$$=\frac{1}{n}\left \{\frac{1}{2}\cdot \left [z(z+2) f^{\prime \prime}(z)\right ]^{(p)}+\frac{1}{n}\cdot \frac{p!}{2\pi i}\int_{\gamma}\frac{n^{2}\left (L_{n}^{*}(f)(v)-f(v)-\frac{v(v+2) f^{\prime \prime}(v)}{2n}\right )}{(v-z)^{p+1}}dv\right \}.$$
Passing to the norm $\|\cdot \|_{r}$, we obtain
$$\|[L_{n}^{*}(f)]^{(p)} - f^{(p)}\|_{r}$$
$$\ge \frac{1}{n}\left \{\frac{1}{2}\left \|\left [e_{1}(e_{1}+2)f^{\prime \prime}\right ]^{(p)}\right \|_{r} -\frac{1}{n}\left \|\frac{p!}{2\pi}\int_{\gamma}\frac{n^{2}\left (L_{n}^{*}(f)(v)-f(v)-\frac{v(v+2) f^{\prime \prime}(v)}{2n}\right )}{(v-z)^{p+1}}dv\right \|_{r}\right \},$$
where by Theorem \ref{t:2}, for all $n > r+2$ it follows
\begin{eqnarray}
&&\left \|\frac{p!}{2\pi}\int_{\gamma}\frac{n^{2}\left (L_{n}^{*}(f)(v)-f(v)-\frac{v(v+2)f^{\prime \prime}(v)}{2n}\right )}{(v-z)^{p+1}}dv\right \|_{r} \nonumber \\
&\le & \frac{p!}{2\pi}\cdot \frac{2\pi r_{1}n^{2}}{K^{p+1}}\left \|L_{n}^{*}(f)-f-\frac{e_{1}(e_1+2)f^{\prime \prime}}{2n}\right \|_{r_{1}} \nonumber \\
&\le &  C_{r, A, M}(f)\cdot \frac{p! r_{1}}{(r_{1}-r)^{p+1}}. \nonumber
\end{eqnarray}
Now, by hypothesis on $f$ we have $\left \|\left [e_{1}(e_{1}+2)f^{\prime \prime}\right ]^{(p)}\right \|_{r}>0$. Indeed, supposing the contrary it follows that $z(z+2)f^{\prime \prime}(z)$ is a polynomial
of degree $\le p-1$, which by the analyticity of $f$ obviously implies that $f$ is a polynomial
of degree $\le p-1$, a contradiction with the hypothesis.

For the rest of the proof, reasoning exactly as in the proof of the above point (i), we immediately get the required conclusion.
\end{proof}

\end{document}